\documentclass[final]{amsart}

\usepackage{amsmath}
\usepackage{amssymb,amsthm}
\usepackage{csquotes}
\usepackage{graphicx}
\usepackage[all]{xy}

\newtheorem{theorem}{Theorem}
\newtheorem{lemma}[theorem]{Lemma}
\newtheorem{proposition}[theorem]{Proposition}

\newtheorem{remark}[theorem]{Remark}
\newtheorem*{itheorem}{Theorem}

\newcommand{\acat}[1]{\mathsf{a}_{#1}}
\newcommand{\cacat}[1]{\widehat{\mathsf{a}}_{#1}}
\newcommand{\lcat}{\mathsf{l}}

\newcommand{\defm}[1]{\mathsf{Def}_{#1}}

\newcommand{\fm}[1]{\mathsf{#1}}

\newcommand{\sets}{\mathsf{Sets}}

\DeclareMathOperator{\fchar}{char}

\DeclareMathOperator{\der}{Der}

\DeclareMathOperator{\enm}{End}
\DeclareMathOperator{\ext}{Ext}

\DeclareMathOperator{\gr}{gr}

\DeclareMathOperator{\hmm}{Hom}
\DeclareMathOperator{\id}{id}
\DeclareMathOperator{\im}{im}
\DeclareMathOperator{\mor}{Mor}
\DeclareMathOperator*{\osum}{\oplus}
\DeclareMathOperator{\rad}{J}

\DeclareMathOperator{\Simp}{Simp}

\begin{document}

\title[The Algebra of Observables]{The Algebra of Observables in
Noncommutative Deformation Theory}

\author{Eivind Eriksen}
\address{BI Norwegian Business School, Department of Economics, N-0442
Oslo, Norway}
\email{eivind.eriksen@bi.no}

\author{Arvid Siqveland}
\address{University of South-Eastern Norway, Faculty of Technology,
Natural Sciences and Maritime Sciences, N-3603 Kongsberg, Norway}
\email{arvid.siqveland@usn.no}


\date{\today}

\begin{abstract}
We consider the algebra $\mathcal O(\fm M)$ of observables and the (formally)
versal morphism $\eta: A \to \mathcal O(\fm M)$ defined by the noncommutative
deformation functor $\defm{\fm M}$ of a family $\fm M = \{ M_1, \dots, M_r
\}$ of right modules over an associative $k$-algebra $A$. By the Generalized
Burnside Theorem, due to Laudal, $\eta$ is an isomorphism when $A$ is finite
dimensional, $\fm M$ is the family of simple $A$-modules, and $k$ is an
algebraically closed field. The purpose of this paper is twofold: First, we
prove a form of the Generalized Burnside Theorem that is more general, where
there is no assumption on the field $k$. Secondly, we prove that the
$\mathcal O$-construction is a closure operation when $A$ is any finitely
generated $k$-algebra and $\fm M$ is any family of finite dimensional
$A$-modules, in the sense that $\eta_B: B \to \mathcal O^B(\fm M)$ is an
isomorphism when $B = \mathcal O(\fm M)$ and $\fm M$ is considered as a
family of $B$-modules.
\end{abstract}

\maketitle

\section{Introduction}

Let $k$ be a field, let $A$ be a finite dimensional associative algebra over
$k$, and let $\fm M = \{ M_1, \dots, M_r \}$ be the family of simple right
$A$-modules, up to isomorphism. We consider the algebra homomorphism
    \[ \rho: A \to \osum_{i=1}^r \enm_k(M_i) \]
given by right multiplication of $A$ on the family $\fm M$. By the extended
version of the classical Burnside Theorem, $\rho$ is surjective when $k$ is
algebraically closed, and if $A$ is semisimple, then it is an isomorphism.
We remark that Artin-Wedderburn theory gives a version of the theorem that
holds over any field:

\begin{itheorem}[Classical Burnside Theorem]
Let $A$ be a finite dimensional $k$-algebra, and let $\{ M_1, \dots, M_r \}$
be the family of simple right $A$-modules. If $\enm_A(M_i) = k$ for $1 \le i
\le r$, then $\rho: A \to \osum_i \, \enm_k(M_i)$ is surjective.
\end{itheorem}

In Laudal \cite{laud02}, a generalization called the Generalized Burnside
Theorem was obtained. This is a structural result for not necessarily
semisimple algebras, and the essential idea of Laudal was to replace $\rho$
with the versal morphism $\eta$ defined by noncommutative deformations of
modules. Let us recall the construction:

Let $A$ be an arbitrary associative $k$-algebra, let $\fm M = \{ M_1, \dots, M_r \}$
be a family of right $A$-modules, and consider the noncommutative deformation
functor $\defm{\fm M}$. This functor has a pro-representing hull $H$ and a
versal family $M_H$ if $\fm M$ is a swarm. Following Laudal \cite{laud02}, we
define the \emph{algebra of observables} of a swarm $\fm M$ to be $\mathcal
O(\fm M) = \enm_H(M_H) \cong ( H_{ij} \otimes_k \hmm_k(M_i,M_j) )$, and its
\emph{versal morphism} to be the algebra homomorphism $\eta: A \to \mathcal
O(\fm M)$ given by right multiplication of $A$ on the versal family $M_H$. It
fits into the commutative diagram
\begin{equation*}
\xymatrix{
A \ar[rrr]^-{\eta} \ar[rrrd]_{\rho} &&& ( H_{ij} \otimes_k \hmm_k(M_i,M_j) )
\ar[d] \\
&&& \osum_{i=1}^r \enm_k(M_i)
}
\end{equation*}
where $\rho: A \to \osum_{i=1}^r \enm_k(M_i)$ is the algebra homomorphism
given by right multiplication of $A$ on the family $\fm M$. By Theorem 1.2
in Laudal \cite{laud02}, it follows that $\eta$ is an isomorphism when $A$
is finite dimensional, $\fm M$ is the family of simple $A$-modules, and $k$
is algebraically closed. In this paper, we prove a more general version of
this result:

\begin{itheorem}[Generalized Burnside Theorem]
Let $A$ be a finite dimensional $k$-algebra, and let $\fm M$ be the family
of simple right $A$-modules, up to isomorphism. The versal morphism $\eta: A
\to \mathcal O(\fm M)$ is injective. If $\enm_A(M_i) = k$ for $1 \le i \le
r$, then $\eta$ is an isomorphism. In particular, $\eta$ is an isomorphism
if $k$ is algebraically closed.
\end{itheorem}

In case $D_i = \enm_A(M_i)$ is a division algebra with $\dim_k D_i > 1$ for
some simple module $M_i$, it is often not difficult to describe the image of
$\eta$ as a subalgebra of $\mathcal O(\fm M)$, and we shall give examples.
As an application of the theorem, we introduce the standard form of any
finite dimensional algbra $A$, given as
    \[ A \cong \mathcal O(\fm M) = ( H_{ij} \otimes_k \hmm_k(M_i,M_j) ) \]
when $\enm_A(M_i) = k$ for $1 \le i \le r$, or as a subalgebra of $\mathcal
O(\fm M)$ in general.

Let $A$ be any finitely generated $k$-algebra and let $\fm M$ be any family
of finite dimensional right $A$-modules. In this more general situation, the
versal morphism $\eta: A \to \mathcal O(\fm M)$ is not necessarily an
isomorphism. However, we may consider the algebra $B = \mathcal O(\fm M)$ of
observables, and $\fm M$ as a family of right $B$-modules, and iterate the
process. We prove that the operation $(A, \fm M) \mapsto (B, \fm M)$ has the
following \emph{closure property}:

\begin{itheorem}[Closure Property]
Let $A$ be a finitely generated $k$-algebra, let $\fm M$ be a family of
finite dimensional $A$-modules, and let $B = \mathcal O(\fm M)$. Then the
versal morphism $\eta^B: B \to \mathcal O^B(\fm M)$ of $\fm M$, considered
as a family of right $B$-modules, is an isomorphism.
\end{itheorem}

One may consider a noncommutative algebraic geometry where the closed points
are represented by simple modules; see for instance Laudal \cite{laud03}.
With this point of view, one may use versal morphisms $\eta: A \to \mathcal
O(\fm M)$ for families $\fm M$ of $A$-modules to construct noncommutative
localization homomorphisms $\eta_s: A \to A_s$ for any $s \in A$. We
explain this construction in Section \ref{s:ncloc}. These localization maps
are universal $S$-inverting localization maps, where $S = \{ 1, s, s^2, \dots
\}$, and can be used as an essential building block for structure sheaves on
noncommutative schemes.

\section{Noncommutative deformations of modules}

Let $A$ be an associative algebra over a field $k$. For any right $A$-module
$M$, there is a \emph{deformation functor} $\defm M: \lcat \to \sets$ defined
on the category $\lcat$ of commutative Artinian local $k$-algebras $R$ with
residue field $k$. We recall that $\defm M(R)$ is the set of equivalence
classes of pairs $(M_R,\tau_R)$, where $M_R$ is an $R$-flat $R$-$A$ bimodule
on which $k$ acts centrally, and $\tau_R: k \otimes_R M_R \to M$ is an
isomorphism of right $A$-modules. Deformations in $\defm M(R)$ are called
\emph{commutative deformations} since the base ring $R$ is commutative.

\emph{Noncommutative deformations} were introduced in Laudal \cite{laud02}.
The deformations considered by Laudal are defined over certain noncommutative
base rings instead of the commutative base rings in $\lcat$. In what follows,
we shall give a brief account of noncommutative deformations of modules. We
refer to Laudal \cite{laud02}, Eriksen \cite{erik06} and Eriksen, Laudal,
Siqveland \cite{erik-laud-siqv17} for further details.

For any positive integer $r$ and any family $\fm M = \{ M_1, \dots, M_r \}$
of right $A$-modules, there is a \emph{noncommutative deformation functor}
$\defm{\fm M}: \acat r \to \sets$, defined on the category $\acat r$ of
noncommutative Artinian $r$-pointed $k$-algebras with exactly $r$ simple
modules (up to isomorphism). We recall that an $r$-pointed $k$-algebra $R$ is
one fitting into a diagram of rings $k^r \to R \to k^r$, where the composition
is the identity. The condition that $R$ has exactly $r$ simple modules holds
if and only if $\overline R \cong k^r$, where $\overline R = R/J(R)$ and
$J(R)$ denotes the Jacobson radical of $R$.

The noncommutative deformations in $\defm{\fm M}(R)$ are equivalence classes
of pairs $(M_R, \tau_R)$, where $M_R$ is an $R$-flat $R$-$A$ bimodule on
which $k$ acts centrally, and $\tau_R: k^r \otimes_R M_R \to M$ is an
isomorphism of right $A$-modules with $M = M_1 \oplus \dots \oplus M_r$. In
concrete terms, an algebra $R$ in $\acat r$ is a matrix ring $R = ( R_{ij} )$
with $R_{ij} = e_i R e_j$. By abuse of notation, we write $e_i$ for the
idempotent $e_i = (0,0, \dots, i, \dots, 0)$ in $k^r$, and also for its image
in $R$ via the structural map $k^r \to R$. As left $R$-modules, we have that
$M_R \cong ( R_{ij} \otimes_k M_j )$ and its right $A$-module structure is
given by an algebra homomorphism
    \[ \eta_R: A \to \enm_R(M_R) \cong ( R_{ij} \otimes_k \hmm_k(M_i,M_j) )
    \]
that lifts $\rho: A \to \osum_i \, \enm_k(M_i)$. Explicitly, we interpret
$\eta_R(a)$ as a right action of $a$ on $M_R$ via
    \[ \eta_R(a) = \sum_i e_i \otimes \rho_i + \sum_{i,j,l} r_{ij}^l \otimes
    \phi_{ij}^l \quad \Longleftrightarrow \quad
    (e_i \otimes m_i)a = e_i \otimes (m_i a) + \sum_{j,l} \, r_{ij}^l \otimes
    \phi_{ij}^l(m_i) \]
where $\rho_i: A \to \enm_k(M_i)$ is the algebra homomorphism given by the
right action of $A$ on $M_i$, such that $\rho = (\rho_1, \dots, \rho_r)$, and
where $r_{ij}^l \in R_{ij}$ and $\phi_{ij}^l \in \hmm_k(M_i,M_j)$. Deformations
in $\defm{\fm M}(R)$ can therefore be represented by commutative diagrams
\begin{equation*}
\xymatrix{
A \ar[rrr]^-{\eta_R} \ar[rrrd]_{\rho} &&& (R_{ij} \otimes_k \hmm_k(M_i,M_j))
\ar[d] \\
&&& \osum_{i=1}^r \enm_k(M_i)
}
\end{equation*}
These deformations are called \emph{noncommutative deformations} since the
base ring $R$ is noncommutative.

For any $r$-pointed algebra $R$, with structural maps $k^r \to R \to k^r$, we
write $I(R) = \ker(R \to k^r)$. Recall that the pro-category $\cacat r$ is the
full subcategory of the category of $r$-pointed algebras consisting of
algebras $R$ such that $R/I(R)^n$ is Artinian for all $n$ and such that $R$ is
complete in the $I(R)$-adic topology.

The family $\fm M = \{ M_1, \dots, M_r \}$ is called a \emph{swarm} if $\dim_k
\ext^1_A(M,M)$ is finite. In this case, the noncommutative deformation functor
$\defm{\fm M}$ has a pro-representing hull $H$ in the pro-category $\cacat r$
and a versal family $M_H \in \defm{\fm M}(H)$; see Theorem 3.1 in Laudal
\cite{laud02}. The defining property of the miniversal pro-couple $(H, M_H)$
is that the induced natural transformation
    \[ \phi: \mor(H,-) \to \defm{\fm M} \]
on $\acat r$ is smooth (which implies that $\phi_R$ is surjective for any $R$
in $\acat r$), and that $\phi_R$ is an isomorphism when $J(R)^2 = 0$. The
miniversal pro-couple $(H, M_H)$ is unique up to (non-canonical) isomorphism.

Let $\fm M$ be a swarm of right $A$-modules, and let $(H,M_H)$ be the
miniversal pro-couple of the noncommtutative deformation functor $\defm{\fm
M}: \acat r \to \sets$. We define the \emph{algebra of observables} of $\fm M$
to be
    \[ \mathcal O(\fm M) = \enm_H(M_H) \cong ( H_{ij} \widehat \otimes_k
    \hmm_k(M_i,M_j) ) \]
where $\widehat{\otimes}$ is the completed tensor product (the completion of
the tensor product), and write $\eta: A \to \mathcal O(\fm M)$ for the induced
\emph{versal morphism}, giving the right $A$-module structure on $M_H$. By
construction, it fits into the commutative diagram
\begin{equation*}
\xymatrix{
A \ar[rrr]^-{\eta} \ar[rrrd]_{\rho} &&& (H_{ij} \widehat \otimes_k
\hmm_k(M_i,M_j)) \ar[d] \\
&&& \osum_{i=1}^r \enm_k(M_i)
}
\end{equation*}

\begin{remark}
Notice that the diagram extends the right action of $A$ on the family $\fm M$
to a right action of $\mathcal O(\fm M)$, such that $\fm M$ is a family of
right $\mathcal O(\fm M)$-modules.
\end{remark}

\begin{remark}
For any $R$ in $\acat r$ and any deformation $M_R \in \defm{\fm M}(R)$, there
is a morphism $u: H \to R$ in $\cacat r$ such that $\defm{\fm M}(u)(M_H) =
M_R$ by the versal property, and the deformation $M_R$ is therefore given by
the composition $\eta_R = u^* \circ \eta$ in the diagram
\begin{equation*}
\xymatrix{
A \ar[rr]^-{\eta} \ar[rrd]_{\eta_R} && \mathcal O(\fm M) \ar[d]^{u^* = u
\otimes \id} \\
&& ( R_{ij} \otimes_k \hmm_k(M_i,M_j) )
}
\end{equation*}
In this sense, the versal morphism $\eta: A \to \mathcal O(\fm M)$ determines
all noncommutative deformations of the family $\fm M$.
\end{remark}

\section{Iterated extensions and injectivity of the versal morphism}
\label{s:mmp} \label{s:gmp}

Let $E$ be a right $A$-module and let $r \ge 1$ be a positive integer. If $E$
has a \emph{cofiltration} of length $r$, given by a sequence
    \[ E = E_r \xrightarrow{f_r} E_{r-1} \to \dots \to E_2 \xrightarrow{f_2}
    E_1 \xrightarrow{f_1} E_0 = 0 \]
of surjective right $A$-module homomorphisms $f_i: E_i \to E_{i-1}$, then we
call $E$ an \emph{iterated extension} of the right $A$-modules $M_1, M_2,
\dots M_r$, where $M_i = \ker(f_i)$. In fact, the cofiltration induces short
exact sequences
    \[ 0 \to M_i \to E_i \xrightarrow{f_i} E_{i-1} \to 0 \]
for $1 \le i \le r$. Hence $E_1 \cong M_1$, $E_2$ is an extension of $E_1$
by $M_2$, and in general, $E_i$ is an extension of $E_{i-1}$ by $M_i$.

Let $\fm M = \{ M_1, \dots, M_r \}$ be a swarm of right $A$-modules, and let
$\defm{\fm M}: \acat r \to \sets$ be its noncommutative deformation functor.
Then $\defm{\fm M}$ has a miniversal pro-couple $(H, M_H)$, and we consider
the induced versal morphism $\eta: A \to \mathcal O(\fm M)$ and its kernel
$K = \ker(\eta)$.

We note that Theorem 3.2 in Laudal \cite{laud02} holds without assumptions
on the base field $k$, since the construction that precedes this theorem
works over any field. From this observation, we obtain the following lemma:

\begin{lemma} \label{l:ek}
Let $\fm M$ be a swarm of right $A$-modules. For any iterated extension $E$
of the family $\fm M$, we have that $E \cdot K = 0$.
\end{lemma}

Let $A$ be a finite dimensional $k$-algebra and let $\fm M$ be the family of
all simple right $A$-modules, up to ismorphism. Then $\fm M$ is a swarm, and
we may consider the versal morphism $\eta: A \to \mathcal O(\fm M)$. If $k$
is algebraically closed, then the versal morphism $\eta$ is injective by
Corollary 3.1 in Laudal \cite{laud02}. Using Lemma \ref{l:ek}, we generalize
this result:

\begin{proposition} \label{p:eta-inj}
If $A$, considered as a right $A$-module, is an iterated extension of a swarm
$\fm M$, then the versal morphism $\eta: A \to \mathcal O(\fm M)$ is
injective. In particular, $\eta$ is injective when $A$ is a finite dimensional
algebra and $\fm M$ is the family of simple right $A$-modules.
\end{proposition}
\begin{proof}
If $A$ is an iterated extension of $\fm M$, then $1 \cdot K = 0$ by Lemma
\ref{l:ek}, and this implies that $K = 0$. If $A$ is finite dimensional, then
the right $A$-module $A$ has finite length, and it is an iterated extension
of the simple modules.
\end{proof}

We remark that our proof, based on Lemma \ref{l:ek}, holds whenever there is
an element $e \in E$ such that $a \mapsto e \cdot a$ defines an injective
right $A$-module homomorphism $A \to E$. This means that $\eta: A \to
\mathcal O(\fm M)$ is injective if there is an iterated extension $E$ of
$\fm M$ such that $E$ contains a copy of $A_A$.

\section{The Generalized Burnside Theorem}

Let $A$ be a finite dimensional $k$-algebra, and let $\fm M = \{ M_1, \dots,
M_r \}$ be the family of simple right $A$-modules, up to isomorphism. Then
$\fm M$ is a swarm, and we consider the versal morphism $\eta: A \to \mathcal
O(\fm M)$ and the commutative diagram
\begin{equation*}
\xymatrix{
A \ar[rrr]^-{\eta} \ar[rrrd]_{\rho} &&& ( H_{ij} \otimes_k \hmm_k(M_i,M_j) )
\ar[d] \\
&&& \osum_{i=1}^r \enm_k(M_i)
}
\end{equation*}
Clearly, $\rho$ factors through $A/\rad(A)$, and if $\enm_A(M_i) = k$ for $1
\le i \le r$, then $A/\rad(A) \to \osum_i \, \enm_k(M_i)$ is an isomorphism
by the Artin-Wedderburn theory for semisimple algebras. This proves the
Classical Burnside Theorem mentioned in the introduction. By Theorem 3.4 in
Laudal \cite{laud02}, the versal morphism $\eta: A \to \mathcal O(\fm M)$ is
an isomorphism when $k$ is algebraically closed. We generalize this result:

\begin{theorem}
Let $A$ be a finite dimensional $k$-algebra and let $\fm M$ be the family of
simple right $A$-modules, up to isomorphism. Then $\eta: A \to \mathcal O(\fm
M)$ is injective, and it is an isomorphism if $\enm_A(M_i) = k$ for $1 \le i
\le r$. In particular, the versal morphism $\eta: A \to \mathcal O(\fm M)$ is
an isomorphism if $k$ is algebraically closed.
\end{theorem}
\begin{proof}
By Proposition \ref{p:eta-inj}, the versal morphism $\eta$ is injective, and
it is enough to prove that $\eta$ is surjective when $\enm_A(M_i) = k$ for $1
\le i \le r$. Note that $\eta$ maps the Jacobson radical $\rad(A)$ of $A$ to
the Jacobson radical $J = (\rad(H)_{ij} \otimes_k \hmm_k(M_i,M_j))$ of
$\mathcal O(\fm M)$. Moreover, $A$ is $\rad(A)$-adic complete since it is
finite dimensional, and $\mathcal O(\fm M)$ is clearly $J$-adic complete. By
a standard result for filtered algebras, it is therefore sufficient to show
that $\gr_1(\eta): J(A)/J(A)^2 \to J/J^2$ is surjective, since $\gr_0(\eta):
A/J(A) \to \osum_i \, \enm_k(M_i)$ is an isomorphism by the Classical
Burnside Theorem. We notice that
    \[ J/J^2 \cong ( ( J(H)/J(H)^2 )_{ij} \otimes_k \hmm_k(M_i,M_j) )
    \cong ( \ext^1_A(M_i,M_j)^* \otimes_k \hmm_k(M_i,M_j) ) \]
since $J(H)/J(H)^2$ is the dual of the tangent space $(\ext^1_A(M_i,M_j))$ of
$\defm{\fm M}$. We note that Lemma 3.7 in Laudal \cite{laud02} holds over
any field. Hence the map
    \[ J(A)/J(A)^2 \to (\ext^1_A(M_i,M_j)^* \otimes_k \hmm_k(M_i,M_j)) \]
induced by $\eta$ is an isomorphism, and this completes the proof.
\end{proof}

\section{The closure property}

Let $A$ be a finitely generated $k$-algebra of the form $A = k\langle x_1,
\dots x_d \rangle/I$, and let $\fm M = \{ M_1, \dots, M_r \}$ be a family of
finite dimensional right $A$-modules. Then $\fm M$ is a swarm, since
    \[ \dim_k \ext^1_A(M_i,M_j) \le \dim_k \der_k(A,\hmm_k(M_i,M_j)) \le
    \dim_k \hmm_k(M_i,M_j)^d \]
The last inequality
follows from the fact that any derivation $D: A \to \hmm_k(M_i,M_j)$ is
determined by $D(x_l) \in \hmm_k(M_i,M_j)$ for $1 \le l \le d$. We consider
the algebra of observables $B = \mathcal O(\fm M)$ of the swarm $\fm M$,
and write $\eta: A \to B$ for its versal morphism. In general, $\fm M = \{
M_1, \dots, M_r \}$ is a family of right $B$-modules via $\eta$.

\begin{lemma}
The family $\fm M = \{ M_1, \dots, M_r \}$ of right $B$-modules is the
simple right $B$-modules, and it is swarm of $B$-modules.
\end{lemma}
\begin{proof}
It follows from the Artin-Wedderburn theory that $\fm M = \{ M_1, \dots, M_r
\}$ is the family of simple modules over
    \[ \overline B = B/J(B) \cong ( H/J(H) \otimes_k \hmm_k(M_i,M_j)) \cong
    \osum_i \, \enm_k(M_i). \]
Since $B$ and $\overline B = B/J(B)$ have the same simple modules, it follows
that $\fm M$ is the family of simple right $B$-modules. We have that
$\ext^1_B(M_i,M_j)$ is a quotient of $\der_k(B,\hmm_k(M_i,M_j))$, and any
derivation $D: B \to \hmm_k(M_i,M_j)$ satisfies $D(J^2) = J D(J) + D(J) J =
0$ when $J = J(B)$ since $\fm M$ is the family of simple $B$-modules. From
the fact that
    \[ B/J^2 \cong ( (H/\rad(H)^2)_{ij} \otimes_k \hmm_k(M_i,M_j) ) \]
is finite dimensional, and in particular a finitely generated $k$-algebra, it
follows from the argument preceding the lemma that $\fm M$ is a swarm of
$B$-modules.
\end{proof}

In this situation, we may iterate the process. Since $\fm M$ is a swarm of
right $B$-modules, the noncommutative deformation functor $\defm{\fm M}^B$
of $\fm M$, considered as a family of right $B$-modules, has a miniversal
pro-couple $(H^B, M_H^B)$. We write $\mathcal O^B(\fm M) = \enm_{H^B}(M_H^B)
\cong ( H_{ij}^B \otimes_k \hmm_k(M_i,M_j) )$ for its algebra of observables
and $\eta^B: B \to \mathcal O^B(\fm M)$ for its versal morphism.

\begin{theorem} \label{t:defm-closure}
Let $A$ be a finitely generated $k$-algebra, let $\fm M = \{ M_1, \dots, M_r
\}$ be a family of finite dimensional $A$-modules, and let $B = \mathcal
O(\fm M)$. Then the versal morphism $\eta^B: B \to \mathcal O^B(\fm M)$ of
$\fm M$, considered as a family of right $B$-modules, is an isomorphism.
\end{theorem}
\begin{proof}
Since $\fm M$ is a swarm of $A$-modules and of $B$-modules, we may consider
the commutative diagram
\begin{equation*}
\xymatrix{
A \ar[r]^-{\eta} \ar[rd]_-{\rho} & B = \mathcal O(\fm M)
\ar[d] \ar[r]^-{\eta^B} & C = \mathcal O^B(\fm M) \ar[dl] \\
& \displaystyle \osum_i \, \enm_k(M_i) &
}
\end{equation*}
The algebra homomorphism $\eta^B$ induces maps $B/\rad(B)^n \to C/\rad(C)^n$
for all $n \ge 1$, and it is enough to show that each of these induced maps
is an isomorphism. For $n = 1$, we have
    \[ B/J(B) \cong C/J(C) \cong \osum_i \, \enm_k(M_i) \]
so it is clearly an isomorphism for $n = 1$. For $n \ge 2$, we have that $B_n
= B/J(B)^n$ is a finite dimensional algebra with the same simple modules as
$B$ since $M_i J^n = 0$. We may therefore consider the versal morphism of
the swarm $\fm M$ of right $B_n$-modules, which is an isomorphism by the
Generalized Burnside Theorem since $\enm_B(M_i) = k$ for $1 \le i \le r$.
Finally, any derivation $D: B \to \hmm_k(M_i,M_j)$ satisfies $D(J^n) = 0$
when $n \ge 2$. Therefore, we have that
    \[ \ext^1_{B_n}(M_i,M_j) \cong \ext^1_B(M_i,M_j) \]
and this implies that $B/J(B)^n \to C/J(C)^n$ coincides with the versal
morphism of the swarm $\fm M$ of right $B_n$-modules. It is therefore an
isomorphism.
\end{proof}

Theorem \ref{t:defm-closure} implies that the assignment $(A, \fm M) \mapsto
(B, \fm M)$ is a closure operation when $A$ is a finitely generated
$k$-algebra and $\fm M = \{ M_1, \dots, M_r \}$ is a family of finite
dimensional right $A$-modules. In other words, the algebra $B = \mathcal
O(\fm M)$ has the following properties:
\begin{enumerate}
    \item The family $\fm M$ is the family of simple right $B$-modules.
    \item The family $\fm M$ has exactly the same module-theoretic
        properties, in terms of extensions and matric Massey products,
        considered as a family of $B$-modules and as a family of $A$-modules.
\end{enumerate}
Moreover, these properties characterize the algebra of observables $B =
\mathcal O(\fm M)$.

\begin{remark}
Assume that $k$ is a field that is not algebraically closed. When $A$ is a
finite dimensional $k$-algebra and $\fm M$ is the family of simple right
$A$-modules, it could happen that the division algebra $D_i = \enm_A(M_i)$
has dimension $\dim_k D_i > 1$ for some simple $A$-modules $M_i$. In this
case, $\eta: A \to \mathcal O(\fm M)$ is not necessarily an isomorphism.
However, if the subfamily $\fm M' = \{ M_i: \enm_A(M_i) = k \} \subseteq
\fm M$ is non-empty, we may consider the algebra $B = \mathcal O(\fm M')$,
and it follows from the closure property that $\eta: B \to \mathcal
O^B(\fm M')$ is an isomorphism. This means that the Generalized Burnside
Theorem holds for the family $\fm M'$ of right $B$-modules.
\end{remark}

\section{Noncommutative localizations via the algebra of observables}
\label{s:ncloc}

Let $A$ be a finitely generated $k$-algebra, and denote by $X = \Simp(A)$ the
set of (isomorphism classes of) simple finite dimensional right $A$-modules.
For any $s \in A$, we write
    \[ D(s) = \{ M \in X: M \xrightarrow{\cdot s} M \text{ is invertible} \}
    \subseteq X. \]
We note that $\{ D(s) \}_{s \in A}$ is a base for a topology on $X$, since
$D(s) \cap D(t) = D(st)$, which we call the \emph{Jacobson topology} on $X =
\Simp(A)$.

For any inclusion $\fm M \subseteq \fm M'$ of finite subsets of $D(s)$, there
is a surjective algebra homomorphism $\mathcal O(\fm M') \to \mathcal O(\fm
M)$. We may consider the algebra homomorphism
    \[ \eta_s: A \to \varprojlim_{\fm M \subseteq D(s)} \mathcal O(\fm M) \]
where the projective limit is taken over all finite subsets $\fm M \subseteq
D(s)$. Notice that $\eta_s(s)$ is a unit, since it is a unit in $\mathcal
O(\fm M)$ for any finite subset $\fm M \subseteq D(s)$. We define $A_s$ to be
the subring of the projective limit
    \[ \varprojlim_{\fm M \subseteq D(s)} \mathcal O(\fm M) \]
generated by $\eta_s(A)$ and $\eta_s(s)^{-1}$. By abuse of notation, we write
$\eta_s$ for the algebra homomorphism $\eta_s: A \to A_s$ into the subring
$A_s$.

Let $S$ be the multiplicative subset $S = \{ 1, s, s^2, \dots \} \subseteq
A$. Then $\eta_s: A \to A_s$ is an $S$-inverting algebra homomorphism, and it
has the following universal property: If $\phi: A \to B$ is any $S$-inverting
algebra homomorphism, then there is a unique algebra homomorphism $\phi_s:
A_s \to B$ such that $\phi_s \circ \eta_s = \phi$. We remark that $A_s$ is a
finitely generated $k$-algebra, generated by the images of the generators of
$A$ and $\eta_s(s)^{-1}$. In general, it is not a (left or right) ring of
fractions.

\section{Applications}

Let $A$ be a finite dimensional $k$-algebra. We consider the family $\fm M =
\{ M_1, \dots, M_r \}$ of simple right $A$-modules. By the Generalized
Burnside Theorem, $A$ can be written in \emph{standard form} as
    \[ A \cong \im(\eta) \subseteq ( H_{ij} \otimes_k \hmm_k(M_i,M_j) ) =
    \mathcal O(\fm M) \]
If $\enm_A(M_i) = k$ for $1 \le i \le r$, then the standard form of $A$ is
$A \cong \mathcal O(\fm M)$, and in general, it is a subalgebra of $\mathcal
O(\fm M)$.

The standard form can, for instance, be used to compare finite dimensional
algebras and determine when they are isomorphic. Let us illustrate this with
a simple example. Let $k$ be a field, and let $A = k[G]$ be the group algebra
of $G = \mathbb Z_3$. In concrete terms, we have that $A \cong k[x]/(x^3-1)$,
and over a fixed algebraic closure $\overline k$ of $k$, we have that
    \[ x^3 - 1 = (x-1)(x^2 + x + 1) = (x-1)(x-\omega)(x-\omega^2) \]
with $\omega \in \overline k$. If $\fchar(k) \neq 3$ and $\omega \in k$, then
the simple $A$-modules are given by $\fm M = \{ M_0, M_1, M_2 \}$, where $M_i
= A/(x-\omega^i)$. Furthermore, a calculation shows that $\ext^1_A(M_i,M_j) =
0$ for $0 \le i,j \le 2$. Hence, the noncommutative deformation functor
$\defm{\fm M}$ has a pro-representing hull $H = k^3$ (it is rigid), and the 
versal morphism $\eta: A \to \mathcal O(\fm M)$ is an isomorphism. The
standard form of $A$ is therefore given by
    \[ A = k[\mathbb Z_3] \cong k^3 = \begin{pmatrix} k & 0 & 0 \\ 0 & k
    & 0 \\ 0 & 0 & k \end{pmatrix}. \]
If $\fchar(k) = 3$, then $M_0$ is the only simple $A$-module since $x^3-1 =
(x-1)^3$, and we find that $\ext^1_A(M_0,M_0) = k$. In this case, it turns
out that $H \cong k[[t]]/(t^3)$, and the standard form of $A$ is given by
$A = k[\mathbb Z_3] \cong k[t]/(t^3)$. In both cases, it follows from the
Generalized Burnside Theorem that $\eta$ is an isomorphism, since $\enm_A(M)
= k$ for all the simple $A$-modules $M$.

If $\fchar(k) \neq 3$ and $\omega \not \in k$, then the simple $A$-modules are 
given by $\fm M = \{ M, N \}$, where $M = M_0 = A/(x-1)$ is $1$-dimensional,
and $N = A/(x^2 + x + 1) \cong k(\omega) = K$ is $2$-dimensional. In this case,
we have that $\enm_A(M) = k$ and $\enm_A(N) = K$, and we find that the
standard form of $A$ is given by
    \[ H = \begin{pmatrix} k & 0 \\ 0 & k \end{pmatrix} \quad \Rightarrow
    \quad A \cong \im(\eta) = \begin{pmatrix} k & 0 \\ 0 & K \end{pmatrix}
    \subseteq \mathcal O(\fm M) = \begin{pmatrix} k && 0 \\ 0 && \enm_k(K)
    \end{pmatrix}. \]
It follows from Proposition \ref{p:eta-inj} that $\eta: A \to \mathcal
O(\fm M)$ is injective. However, it is not an isomorphism in this case.

\bibliographystyle{plain}
\bibliography{eeriksen}

\end{document}